\documentclass[12pt, leqno]{amsart}

\usepackage{graphics}
\usepackage{amssymb,amsmath,amsthm,amscd}
\usepackage{mathrsfs}
\usepackage{enumerate}
\usepackage{color}
\usepackage[all]{xy}

\numberwithin{equation}{section}

\newcommand{\C}{{\mathbb C}}
\newcommand{\Q}{\mathbb {Q}}

\newcommand{\Z}{{\mathbb Z}}
\newcommand{\B}{{\mathbf{B}}}
\newcommand{\A}{{\mathbf A}}

\newcommand{\R}{{\rm R}}

\newcommand{\CC}{{\mathscr{C}}}

\newcommand{\seteq}{\mathbin{:=}}

\theoremstyle{plain}
\newtheorem{lemma}{Lemma}[section]
\newtheorem{prop}[lemma]{Proposition}
\newtheorem{theorem}[lemma]{Theorem}
\newcommand{\Prop}{\begin{prop}}
\newcommand{\enprop}{\end{prop}}
\newcommand{\Lemma}{\begin{lemma}}
\newcommand{\enlemma}{\end{lemma}}
\newcommand{\Th}{\begin{theorem}}
\newcommand{\enth}{\end{theorem}}
\newtheorem{corollary}[lemma]{Corollary}
\newcommand{\Cor}{\begin{corollary}}
\newcommand{\encor}{\end{corollary}}
\newtheorem{definition}[lemma]{Definition}

\newcommand{\Def}{\begin{definition}}
\newcommand{\edf}{\end{definition}}

\theoremstyle{definition}
\newtheorem{remark}[lemma]{Remark}
\newtheorem{example}[lemma]{Example}

\newcommand{\g}{{\mathfrak{g}}}

\newcommand{\Hom}{\operatorname{Hom}}
\newcommand{\End}{\operatorname{End}}

\newcommand{\isoto}[1][]{\mathop{\xrightarrow[#1]%
{{\raisebox{-.6ex}[0ex][-.6ex]{$\mspace{2mu}\sim\mspace{2mu}$}}}}}

\newcommand{\eq}{\begin{eqnarray}}
\newcommand{\eneq}{\end{eqnarray}}

\newcommand{\eqn}{\begin{eqnarray*}}
\newcommand{\eneqn}{\end{eqnarray*}}

\newcommand{\on}{\operatorname}
\newcommand{\Ker}{\on{Ker}}

\newcommand{\bni}{\be[{\rm(i)}]}
\newcommand{\bna}{\be[{\rm(a)}]}

\newcommand{\QED}{\end{proof}}
\newcommand{\Proof}{\begin{proof}}

\newcommand{\soplus}{\mathop{\mbox{\normalsize$\bigoplus$}}\limits}

\newcommand{\To}[1][\phantom{aaaa}]{\xrightarrow{\,#1\,}}
\newcommand{\id}{\on{id}}
\newcommand{\ba}{\begin{array}}
\newcommand{\ea}{\end{array}}

\newcommand{\bi}{\begin{enumerate}[{\rm(i)}]}

\newcommand{\set}[2]{\left\{#1 \mathbin{;}\, #2 \right\}}

\newcommand{\Mod}{\operatorname{Mod}}

\newcommand{\hs}{\hspace*}

\newcommand{\eqsub}{\begin{subequations}\begin{eqnarray}}
\newcommand{\eneqsub}{\end{eqnarray}\end{subequations}}

\newcommand{\ol}{\overline}

\newcommand{\nc}{\newcommand}
\nc{\la}{\lambda}
\nc{\lam}{\lambda}
\nc{\U}[1][\g]{U_q(#1)}
\nc{\te}{\tilde{e}}
\nc{\tei}{\tilde{e}_i}
\nc{\tf}{\tilde{f}}
\nc{\tfi}{\tilde{f}_i}
\nc{\tU}{\widetilde U_q(\g)}
\nc{\tE}{\tilde{E}}
\nc{\tF}{\tilde{F}}

\nc{\tk}{\tilde{k}}
\nc{\tkone}{\tk_{\ol{1}}}
\nc{\teone}{\tilde{e}_{\ol{1}}}
\nc{\tfone}{\tilde{f}_{\ol{1}}}

\nc{\teibar}{\tilde{e}_{\ol{i}}} \nc{\tfibar}{\tilde{f}_{\ol{i}}}
\nc{\tki}{{\tk}_{\ol {i}}}

\nc{\BZ}{{\mathbb{Z}}}
\nc{\al}{\alpha}
\nc{\qs}{{q}}
\nc{\lan}{\langle}
\nc{\ran}{\rangle}
\nc{\re}{{\mathrm{re}}}
\nc{\wt}{\operatorname{wt}}
\nc{\ch}{\operatorname{ch}}
\nc{\Uf}[1][\g]{U^-_q(#1)}
\nc{\Ue}{U^+_q(\g)}
\nc{\eps}{\varepsilon}
\nc{\vphi}{\varphi}
\nc{\sphi}{\varphi^*}
\nc{\seps}{\varepsilon^*}

\nc{\nn}{\nonumber}
\def\max{{\mathop{\mathrm{max}}}}
\nc{\vp}{\varpi}
\nc{\cls}{{\operatorname{cl}}}
\nc{\Wt}{{\operatorname{Wt}}}
\nc{\Us}{U'_q(\g)}
\nc{\La}{\Lambda}
\nc{\ro}{{\rm(}}
\nc{\rf}{{\rm)}}
\nc{\norm}{{\mathrm{norm}}}
\nc{\qbox}{\quad\mbox}
\nc{\braid}{{\mathfrak{B}}}
\nc{\Ad}{\operatorname{Ad}}
\nc{\Aut}{\operatorname{Aut}}
\nc{\dt}[1]{\tilde{\tilde #1}}
\nc{\Sn}{S^{{\mathrm{norm}}}}
\nc{\aff}{{\mathrm{aff}}}
\nc{\rk}{{\mathrm{rk}}}
\nc{\tQ}{\widetilde{Q}}
\nc{\tP}{\widetilde{P}}
\nc{\tW}{\widetilde{W}}
\nc{\Dyn}{\mathrm{Dyn}}
\nc{\tD}{\widetilde{\Delta}}
\nc{\height}{{\operatorname{ht}}}
\nc{\bl}{\bigl(}
\nc{\br}{\bigr)}
\nc{\Hecke}{\mathrm{H}}
\nc{\HA}{\Hecke^{\mathrm{A}}}
\nc{\HB}{\Hecke^{\mathrm{B}}}
\newcommand{\scbul}{{\,\raise1pt\hbox{$\scriptscriptstyle\bullet$}\,}}
\nc{\vac}{{\phi}}
\nc{\Bt}{\B_\theta(\g)}
\nc{\be}{\begin{enumerate}}
\nc{\ee}{\end{enumerate}}
\nc{\low}{{\mathrm{low}}}
\nc{\upper}{{\mathrm{up}}}
\nc{\Zodd}{\Z_{\mathrm{odd}}}
\nc{\Ft}[1][n]{\mathbb{P}\mathrm{ol}_{#1}}
\nc{\Ftf}[1][n]{\widetilde{\mathbb{P}\mathrm{ol}}_{#1}}
\nc{\KA}{\on{K}^{\mathrm{A}}}
\nc{\KB}{\on{K}^{\mathrm{B}}}
\nc{\Res}{\on{Res}}
\nc{\Fc}[1][{n,m}]{\mathbf{F}_{#1}}
\nc{\tphi}{\tilde{\varphi}}
\nc{\CO}{\mathscr{O}}
\nc{\inte}{\mathrm{int}}
\nc{\Oint}{\mathcal{O}^{\ge0}_{\inte}}
\nc{\vs}{\vspace}
\nc{\tL}{\widetilde{L}}
\nc{\noi}{\noindent}
\nc{\heigh}{\mathfrak{t}}
\nc{\lowest}{\mathfrak{l}}
\nc{\rootl}{\mathsf{Q}}
\nc{\cl}{{\rm{cl}}}
\nc{\uqpg}{U'_q(\mathfrak g)}
\nc{\Ohat}{\widehat{\mathcal{O}}}

\nc{\KLR}{Khovanov-Lauda-Rouquier algebra}
\nc{\KLRs}{Khovanov-Lauda-Rouquier algebras}
\nc{\cor}{\mathbf{k}}
\nc{\cora}{{\cor(A)}}
\nc{\haut}{\mathrm{ht}}
\nc{\tens}{\mathop\otimes}
\nc{\gmod}{\mbox{-$\mathrm{gmod}$}}
\nc{\proj}{\mbox{-$\mathrm{proj}$}}
\nc{\modl}{\mbox{-$\mathrm{mod}$}}

\nc{\h}{\mathfrak h}
\nc{\Rnorm}{R^{\rm{\scriptsize{norm}}}}
\nc{\K}{\C(q)}
\nc{\Vhat}{\widehat{V}}
\nc{\F}{\mathcal{F}}

\nc{\fd}[1][A]{\on{\mathrm{flat.dim}_{#1}}}
\nc{\bP}[1][n]{\mathbb{P}_{#1}}
\nc{\bPh}[1][n]{\widehat{\mathbb{P}}_{#1}}
\nc{\bK}[1][n]{\widehat{\mathbb{K}}_{#1}}
\nc{\bV}[1][n]{\widehat{V}^{\otimes{#1}}}
\nc{\bVK}[1][n]{\widehat{V}^{#1}_K}
\nc{\opp}{\mathrm{opp}}

\newlength{\mylength}
\title[R-matrices and KLR algebras]%
{R-matrices for quantum affine algebras and Khovanov-Lauda-Rouquier algebras, I}

\author[S.-J. Kang, M. Kashiwara, M. Kim]{Seok-Jin Kang$^{1}$, 
Masaki Kashiwara$^{2}$, Myungho Kim $^{3}$
}

\address{Department of Mathematical Sciences
         and
         Research Institute of Mathematics \\
         Seoul National University \\ Seoul 151-747, Korea}

         \email{sjkang@math.snu.ac.kr}

\address{Research Institute for Mathematical Sciences \\
          Kyoto University \\ Kyoto 606-8502, Japan \\
          \& Department of Mathematical Sciences
         and
         Research Institute of Mathematics \\
         Seoul National University \\ Seoul 151-747, Korea}

         \email{masaki@kurims.kyoto-u.ac.jp}

\address{School of Mathematics, Korea Institute for Advanced Study \\ Seoul 130-722, Korea}
         \email{mhkim@kias.re.kr}

\thanks{$^{1}$This work was partially supported by NRF Grant \# 2012-005700.}
\thanks{$^{2}$This work was partially supported by Grant-in-Aid for
Scientific Research (B) 22340005, Japan Society for the Promotion of Science.}
\thanks{$^{3}$ This work was partially supported by NRF Grant \# 2011-0027952.}



\begin{document}

\begin{abstract}
Let $I$ be a finite set of pairs consisting of good
$U'_q(\g)$-modules and invertible elements in the base field
$\C(q)$. The distribution of poles of normalized $R$-matrices yields
Khovanov-Lauda-Rouquier algebras $R^{I}(n)$ for $n \ge 0$.
We define a functor ${\mathcal F}$ from the category of 
finite-dimensional $R^I(n)$-modules to the category of finite-dimensional
$U_q'(\g)$-modules. We show that the functor ${\mathcal F}$ sends
 convolution products to tensor products and is exact if $R^I(n)$ is of type $A$,$D$,$E$.
\end{abstract}

\maketitle

\section*{Introduction}

Let $U_q(\g)=U_q(\widehat{\mathfrak{sl}_{N}})$ be the quantum affine
algebra of type $A_{N-1}^{(1)}$ and let $V_{\rm aff}$ be the
affinization of $V$, the vector representation of $U'_{q}(\g)$. We
denote by $H_{n}^{\rm aff}$ the affine Hecke algebra. In \cite{CP96,
Che, GRV94}, a functor
$$\mathcal{F}_n \colon H_{n}^{\rm aff} \modl
\rightarrow U'_{q}(\g) \modl$$
is introduced, which is given by
$$M \longmapsto (V_{\rm aff})^{\otimes n}\otimes_{H_{n}^{\rm aff}} M.$$
It was shown that the functor ${\mathcal F}_{n}$ is exact and that
${\mathcal F}_{n}$ is fully faithful when $N > n$.

In \cite{Kim12}, this idea was generalized to the case when
$U'_q(\g)$ is an arbitrary quantum affine algebra and $V$ is a good
module. That is, the poles of the $R$-matrix on $(V_{\rm aff})^{\otimes n}$ 
define a quiver, which in turn defines a 
Khovanov-Lauda-Rouquier algebra $R(n)$. Then one can construct a
functor $${\mathcal F}_n \colon R(n) \modl \rightarrow U'_q(\g) \modl $$
given by
$$M \longmapsto \bV\otimes_{R(n)} M,$$
where $\bV$ is a certain completion of $(V_{\rm aff})^{\otimes n}$. 
Moreover, it was shown that ${\mathcal F}_{n}$
sends convolution products to tensor products. When
$U'_q(\g)=U'_q(\widehat{\mathfrak{sl}_{N}})$ and $V$ is the vector
representation, it coincides with the one given in \cite{CP96}.

In this article, we extend this construction to the case when
we have a family of good modules. More precisely, let $\{V_s\}_{s\in
\mathcal{S}}$ be a family of good $\uqpg$-modules and let $I$ be a
finite subset of $\mathcal{S}\times \C(q)^{\times}$.
An element $i \in I$ is denoted by $i=({S(i)}, X(i))$.

Then we can define a quiver $\Gamma_I$ with $I$ 
as a set of vertices as follows.
For $i,j\in I$, let 
$\Rnorm_{V_{S(i)},V_{S(j)}}\colon (V_{S(i)})_u\tens (V_{S(j)})_v
\to  (V_{S(j)})_v\tens(V_{S(i)})_u$ be the normalized R-matrix.
We join $i$ and $j$ by edges if
$\Rnorm_{V_{S(i)},V_{S(j)}}$ has a pole at $u/v=X(i)/X(j)$.
Then the quiver $\Gamma_I$ 
defines a  Khovanov-Lauda-Rouquier algebra $R^I(n)$. For each sequence
$\nu \in I^n$, set
$$V_\nu = (V_{S(\nu_1)})_{\rm aff} \otimes \cdots \otimes (V_{S(\nu_n)})_{\rm aff},$$
and let $\bV$ be a certain completion of $\bigoplus_{\nu \in I^n} V_{\nu}$.
Then $\bV$ has a structure of a 
$(U'_q(\g), R^I(n))$-bimodule, and we can define the functor
$$\mathcal{F}_n \colon R^I(n) \modl \rightarrow U'_q(\g) \modl$$
given by
$$M \longmapsto \bV \otimes_{R^I(n)} M.$$

In this paper, we first  prove that ${\mathcal F}_{n}$ sends convolution products
to tensor products (Theorem \ref{thm:conv to tensor}). Moreover,
when the Cartan datum associated with $I$ is of type $A$, $D$, $E$,
we show that ${\mathcal F}_{n}$ is an exact functor (Theorem
\ref{thm:exact}).

In a forthcoming paper, we will study several applications of the
above functor ${\mathcal F}_{n}$. In particular, we will give a
construction of categories whose Grothendieck groups have quantum
cluster structures.
We will also provide an interpretation of the isomorphism between
$t$-deformed Grothendieck rings and negative parts of the quantum
groups, which was established recently in \cite{HL11}.

\section{Quantum groups and Khovanov-Lauda-Rouquier algebras}

\subsection{Quantum groups}

In this section, we recall the definitions of the quantum groups.
Let $I$ be a finite index set.
A \emph{Cartan datum} is a quintuple $(A,P, \Pi,P^{\vee},\Pi^{\vee})$ consists of
\begin{enumerate}[(a)]
\item an integer-valued matrix $A=(a_{ij})_{i,j \in I}$,
called the \emph{symmetrizable generalized Cartan matrix},
 which satisfies
\bni
\item $a_{ii} = 2$ $(i \in I)$,
\item $a_{ij} \le 0 $ $(i \neq j)$,
\item $a_{ij}=0$ if $a_{ji}=0$ $(i,j \in I)$,
\item there exists a diagonal matrix
$D=\text{diag} (\mathsf s_i \mid i \in I)$ such that $DA$ is
symmetric, and $\mathsf s_i$ are positive integers.
\end{enumerate}

\item a free abelian group $P$ of finite rank, called the \emph{weight lattice},
\item $\Pi= \{ \alpha_i \in P \mid \ i \in I \}$, called
the set of \emph{simple roots},
\item $P^{\vee}\seteq\Hom(P, \Z)$, called the \emph{dual weight lattice},
\item $\Pi^{\vee}= \{ h_i \ | \ i \in I \}\subset P^{\vee}$, called
the set of \emph{simple coroots},
\end{enumerate}

satisfying the following properties:
\begin{enumerate}
\item[(i)] $\langle h_i,\alpha_j \rangle = a_{ij}$ for all $i,j \in I$,
\item[(ii)] $\Pi$ is linearly independent,
\item[(iii)] for each $i \in I$, there exists $\Lambda_i \in P$ such that
           $\langle h_j, \Lambda_i \rangle =\delta_{ij}$ for all $j \in I$.
\end{enumerate}
We call $\Lambda_i$ the \emph{fundamental weights}.
The free abelian group $\rootl\seteq\soplus_{i \in I} \Z \alpha_i$ is called the
\emph{root lattice}. Set $\rootl^{+}= \sum_{i \in I} \Z_{\ge 0}
\alpha_i\subset\rootl$ and $\rootl^{-}= \sum_{i \in I} \Z_{\le0}
\alpha_i\subset\rootl$. For $\beta=\sum_{i\in I}m_i\al_i\in\rootl$,
we set
$\haut(\beta)=\sum_{i\in I}|m_i|$.

Set $\mathfrak{h}=\Q \otimes_\Z P^{\vee}$.
Then there exists a symmetric bilinear form $(\quad|\quad)$ on
$\mathfrak{h}^*$ satisfying
$$ (\alpha_i | \alpha_j) =\mathsf s_i a_{ij} \quad (i,j \in I)
\quad\text{and $\lan h_i,\lambda\ran=
\dfrac{2(\alpha_i|\lambda)}{(\alpha_i|\alpha_i)}$ for any $\lambda\in\mathfrak{h}^*$ and $i \in I$}.$$

Let $q$ be an indeterminate. For each $i \in I$, set $q_i = q^{\mathsf s_i}$.

\begin{definition} \label{def:qgroup}
The {\em quantum group} $U_q(\g)$ associated with a Cartan datum
$(A,P,\Pi,P^{\vee}, \Pi^{\vee})$ is the associative algebra over
$\mathbb Q(q)$ with $1$ generated by $e_i,f_i$ $(i \in I)$ and
$q^{h}$ $(h \in P^{\vee})$ satisfying following relations:
\begin{equation*}
\begin{aligned}
& q^0=1,\ q^{h} q^{h'}=q^{h+h'} \ \ \text{for} \ h,h' \in P^{\vee},\\
& q^{h}e_i q^{-h}= q^{\langle h, \alpha_i \rangle} e_i, \ \
          \ q^{h}f_i q^{-h} = q^{-\langle h, \alpha_i \rangle} f_i \ \ \text{for} \ h \in P^{\vee}, i \in
          I, \\
& e_if_j - f_je_i = \delta_{ij} \dfrac{K_i -K^{-1}_i}{q_i- q^{-1}_i
}, \ \ \mbox{ where } K_i=q^{\mathsf s_ih_i}, \\
& \sum^{1-a_{ij}}_{r=0} (-1)^r \left[\begin{matrix}1-a_{ij}
\\ r\\ \end{matrix} \right]_i e^{1-a_{ij}-r}_i
         e_j e^{r}_i =0 \quad \text{ if } i \ne j, \\
& \sum^{1-a_{ij}}_{r=0} (-1)^r \left[\begin{matrix}1-a_{ij}
\\ r\\ \end{matrix} \right]_i f^{1-a_{ij}-r}_if_j
        f^{r}_i=0 \quad \text{ if } i \ne j.
\end{aligned}
\end{equation*}
\end{definition}

Here, we set $[n]_i =\dfrac{ q^n_{i} - q^{-n}_{i} }{ q_{i} - q^{-1}_{i} },\quad
  [n]_i! = \prod^{n}_{k=1} [k]_i$ and
  $\left[\begin{matrix}m \\ n\\ \end{matrix} \right]_i= \dfrac{ [m]_i! }{[m-n]_i! [n]_i! }\;$
  for each $m,n \in Z_{\geq 0}$, $i \in I$.

We denote by $-$ the $\C$-algebra automorphism of $U_q(\g)$ given by
\begin{align*}
\overline q = q^{-1}, && \overline{q^h} = q^{-h} \ \ (h \in P^{\vee}), && \overline{e_i} = e_i \ (i \in I), && \overline{f_i} = f_i \ (i \in I).
\end{align*}

We denote by $\varphi$ and $*$ the $\mathbb Q(q)$-algebra anti-automorphisms of $U_q(\g)$ given by
\begin{align*}
(q^h)^* = q^{-h} \ \ (h \in P^{\vee}), && {e_i}^* = e_i \ (i \in I), && {f_i}^* = f_i \ (i \in I), \\
\varphi(q^h) = q^{h} \ \ (h \in P^{\vee}), && \varphi(e_i) = f_i \ (i \in I), && \varphi(f_i) = e_i \ (i \in I).
\end{align*}

We have two comultiplications $ \Delta_{\pm} \colon U_q(\g) \rightarrow U_q(\g) \otimes U_q(\g) $ given by
 \begin{align*}
 & \Delta_+(q^h) = q^h \otimes q^h, &
 & \Delta_+(e_i) = e_i \otimes 1 + K_i \otimes e_i, &
 & \Delta_+(f_i) = f_i \otimes K_i^{-1} + 1 \otimes f_i,
  \end{align*}
 \begin{align*}
 & \Delta_-(q^h) = q^h \otimes q^h, &
 & \Delta_-(e_i) = e_i \otimes K_i^{-1} + 1 \otimes e_i, &
 & \Delta_-(f_i) = f_i \otimes 1 + K_i \otimes f_i.
  \end{align*}

Let $U_q^{+}(\g)$ (resp.\ $U_q^{-}(\g)$) be the subalgebra of
$U_q(\g)$ generated by $e_i$'s (resp.\ $f_i$'s), and let $U^0_q(\g)$
be the subalgebra of $U_q(\g)$ generated by $q^{h}$ $(h \in
P^{\vee})$. Then we have the \emph{triangular decomposition}
$$ U_q(\g) \simeq U^{-}_q(\g) \otimes U^{0}_q(\g) \otimes U^{+}_q(\g),$$
and the {\em weight space decomposition}
$$U_q(\g) = \bigoplus_{\beta \in Q} U_q(\g)_{\beta},$$
where $U_q(\g)_{\beta}\seteq\set{ x \in U_q(\g)}{q^{h}x q^{-h}
=q^{\langle h, \beta \rangle}x \text{ for any } h \in P^{\vee}}$.

Let $\A= \Z[q, q^{-1}]$ and set
$$e_i^{(n)} = e_i^n / [n]_i!, \quad f_i^{(n)} =
f_i^n / [n]_i! \ \ (n \in \Z_{\ge 0}).$$
We define the $\A$-form
$U_{\A}(\g)$ to be the $\A$-subalgebra of $U_q(\g)$ generated by
$e_i^{(n)}$, $f_i^{(n)}$ $(i \in I, n \in \Z_{\ge 0})$, $q^h$ ($h\in P^\vee$).
Let $U_{\A}^{+}(\g)$ (resp.\ $U_{\A}^{-}(\g)$) be the
$\A$-subalgebra of $U_q(\g)$ generated by $e_i^{(n)}$ (resp.\
$f_i^{(n)}$) for $i\in I$, $n \in \Z_{\ge 0}$.

\subsection{Crystal bases and global bases}
\hfill

Let $e_i'$ and $e_i''$ be the operator on $U_q^-(\g)$ defined by
\begin{align*}
  [e_i, x] = \frac{e_i'' (x)K_i - K_i^{-1}e_i'(x)}{q_i-q_i^{-1}} & \quad (x \in U_q^-(\g)).
\end{align*}

The operator $e_i'$ satisfies
\begin{align*}
e_i'(f_j) =\delta_{ij} & \quad (i,j \in I) \ \text{and} && e_i'(xy) = e_i'(x) y + q_i^{\langle h_i, \wt(x)\rangle} x e_i'(y) & \quad (x,y \in U_q^-(\g)).
\end{align*}
Hence we have
\begin{align*}
e_i' f_j = q_i^{- \langle h_i, \alpha_j\rangle}f_j e_i'+ \delta_{ij},
\end{align*}
where $f_j \in \End_{\mathbb Q(q)}(U_q^-(\g))$ denotes the left multiplication by $f_j \in U_q^-(\g)$.
There exists a unique non-degenerate symmetric bilinear form $(\cdot,\cdot)_-$ on $U_q^-(\g)$ satisfying
\begin{align*}
(1,1)_-=1 \ \text{and} & \quad (e_i'x,y )_-=(x, f_i y)_- \ \text{for any} \ x,y \in U_q^-(\g).
\end{align*}
Any element $x \in U_q^-(\g)$ can be uniquely written as
\begin{align*}
  x = \sum_{n \geq 0} f_i^{(n)} x_n & \quad \text{for some} \ x_n \in \Ker(e_i').
\end{align*}
We define the Kashiwara operators $\tilde e_i$ and $\tilde f_i$ on $U_q^-(\g)$ by
\begin{align*}
  \tilde e_i x = \sum_{n \geq 1} f_i^{(n-1)} x_n, & \quad
  \tilde f_i x = \sum_{n \geq 0} f_i^{(n+1)} x_n.
\end{align*}

\begin{prop}[\cite{Kas91}]
Let $\A_0 = \{ f \in \mathbb Q(q) \,|\, f \ \text{is regular at} \ q=0\}$.
  Define
  \begin{align*}
    L(\infty) &= \sum_{\ell \geq 0, i_1,\ldots, i_\ell \in I} \A_0 \tilde f_{i_1} \cdots \tilde f_{i_\ell} \cdot 1 \subset U_q^-(\g), \\
    B(\infty) &= \{ \tilde f_{i_1} \cdots \tilde f_{i_\ell} \cdot 1 \ \text{mod} \ qL(\infty) \,|\,
    \ell \geq 0, i_1,\ldots, i_\ell \in I\ \} \subset L(\infty) / q L(\infty).
  \end{align*}
  Then we have
  \bna
    \item $\tilde e_i L(\infty) \subset L(\infty)$ and $\tilde f_i L(\infty) \subset L(\infty)$,
    \item $B(\infty)$ is a basis of $L(\infty) / q L(\infty)$,
    \item $\tilde f_i B(\infty) \subset B(\infty)$ and $\tilde e_i B(\infty) \subset B(\infty) \cup \{0\}$,
    \ee
    We call the pair $(L(\infty), B(\infty))$ the {\it crystal basis of $U_q^-(\g)$}.
\end{prop}

Set $L(\infty)^- =\{\overline x \, | \, x \in L(\infty)\}$.
Then the triple $(L(\infty), L(\infty)^-, U_\A^-(\g))$ is {\it balanced}; i.e.,
\begin{align*}
  L(\infty) \cap L(\infty)^- \cap U_\A^-(\g) \rightarrow L(\infty) / q L(\infty)
\end{align*}
is a $\mathbb Q$-vector space isomorphism.

Let $G^{\text{low}}$ be the inverse of the above isomorphism.
Then $$\mathbf B^{\text{low}}\seteq \set{G^{\text{low}}(b)}{b \in B(\infty)}$$ forms a basis of $U_q^-(\g)$.
We call it the {\it lower global basis}.
Let
$$\mathbf B^{\text{up}}\seteq \set{G^{\text{up}}(b)}{b \in B(\infty)}$$
 be the dual basis of $\mathbf B^{\text{low}}$ with
respect to the bilinear form $(\cdot,\cdot)_-$.
It is called the {\it upper global basis of $U_q^-(\g)$}.

\bigskip
\subsection{\KLRs\ }
 \hfill

Now we recall the definition of the \KLRs\ associated with a given Cartan datum $(A, P, \Pi, P^{\vee}, \Pi^{\vee})$.

Let $\cor$ be a base field.
For $i,j\in I$ such that $i\not=j$, set
$$S_{i,j}=\set{(p,q)\in\Z_{\ge0}^2}{(\al_i | \al_i)p+(\al_j | \al_j)q=-2(\al_i | \al_j)}.
$$
Let us define the polynomials $(Q_{ij})_{i,j\in I}$ in $\cor[u,v]$
by
\begin{equation} \label{eq:Q}
Q_{ij}(u,v) = \begin{cases}\hs{5ex} 0 \ \ & \text{if $i=j$,} \\
\sum\limits_{(p,q)\in S_{i,j}}
t_{i,j;p,q} u^p v^q\quad& \text{if $i \neq j$.}
\end{cases}
\end{equation}
They satisfy $t_{i,j;p,q}=t_{j,i;q,p}$ (equivalently, $Q_{i,j}(u,v)=Q_{j,i}(v,u)$) and
$t_{i,j:-a_{ij},0} \in \cor^{\times}$.

We denote by
$S_{n} = \langle s_1, \ldots, s_{n-1} \rangle$ the symmetric group
on $n$ letters, where $s_i\seteq (i, i+1)$ is the transposition of $i$ and $i+1$.
Then $S_n$ acts on $I^n$ by place permutations.

\begin{definition}
The {\em Khovanov-Lauda-Rouquier algebras $R(n)$ of degree $n$
associated with the Cartan datum $(A,P, \Pi,P^{\vee},\Pi^{\vee})$ and the matrix $(Q_{ij})_{i,j \in I}$} is the associative algebra
over $\cor$ generated by the elements $\{ e(\nu) \}_{\nu \in I^{n}}$, $ \{x_k \}_{1 \le k
\le n}$, $\{ \tau_m \}_{1 \le m \le n-1}$ satisfying the following
defining relations:
\begin{equation*} \label{eq:KLR}
\begin{aligned}
& e(\nu) e(\nu') = \delta_{\nu, \nu'} e(\nu), \ \
\sum_{\nu \in I^{n}} e(\nu) = 1, \\
& x_{k} x_{m} = x_{m} x_{k}, \ \ x_{k} e(\nu) = e(\nu) x_{k}, \\
& \tau_{m} e(\nu) = e(s_{m}(\nu)) \tau_{m}, \ \ \tau_{k} \tau_{m} =
\tau_{m} \tau_{k} \ \ \text{if} \ |k-m|>1, \\
& \tau_{k}^2 e(\nu) = Q_{\nu_{k}, \nu_{k+1}} (x_{k}, x_{k+1})
e(\nu), \\
& (\tau_{k} x_{m} - x_{s_k(m)} \tau_{k}) e(\nu) = \begin{cases}
-e(\nu) \ \ & \text{if} \ m=k, \nu_{k} = \nu_{k+1}, \\
e(\nu) \ \ & \text{if} \ m=k+1, \nu_{k}=\nu_{k+1}, \\
0 \ \ & \text{otherwise},
\end{cases} \\
& (\tau_{k+1} \tau_{k} \tau_{k+1}-\tau_{k} \tau_{k+1} \tau_{k}) e(\nu)\\
& =\begin{cases} \dfrac{Q_{\nu_{k}, \nu_{k+1}}(x_{k},
x_{k+1}) - Q_{\nu_{k}, \nu_{k+1}}(x_{k+2}, x_{k+1})} {x_{k} -
x_{k+2}}e(\nu) \ \ & \text{if} \
\nu_{k} = \nu_{k+2}, \\
0 \ \ & \text{otherwise}.
\end{cases}
\end{aligned}
\end{equation*}
\end{definition}

The above relations are homogeneous provided with
\begin{equation*} \label{eq:Z-grading}
\deg e(\nu) =0, \quad \deg\; x_{k} e(\nu) = (\alpha_{\nu_k}
| \alpha_{\nu_k}), \quad\deg\; \tau_{l} e(\nu) = -
(\alpha_{\nu_l} | \alpha_{\nu_{l+1}}),
\end{equation*}
and hence $R(n)$ is a ($\Z$-)graded algebra.
 For a graded $R(\beta)$-module $M=\bigoplus_{k \in \Z} M_k$, we define 
$qM =\bigoplus_{k \in \Z} (qM)_k$, where
 \begin{align*}
 (qM)_k = M_{k-1} & \ (k \in \Z).
 \end{align*}
We call $q$ the \emph{grade-shift functor} on the category of graded $R(n)$-modules.

For $n \in \Z_{\geq 0}$ and $\beta \in Q_+$ such that $|\beta| = n$, we set
$$I^{\beta} = \set{\nu = (\nu_1, \ldots, \nu_n) \in I^{n}}%
{ \alpha_{\nu_1} + \cdots + \alpha_{\nu_n} = \beta },$$ and
\begin{equation*}
R(\beta) = \bigoplus_{\nu \in I^{\beta}} R(n) e(\nu).
\end{equation*}
Note that for each $\beta \in \rootl^+ $, 
the element $e(\beta)\seteq\sum_{\nu \in I^{\beta}} e(\nu)$ is 
a central idempotent of $R(n)$.
The algebra $R(\beta)$ is called the {\it Khovanov-Lauda-Rouquier algebra at $\beta$}.

For a graded $R(m)$-module $M$ and a graded $R(n)$-module $N$,
we define the \emph{convolution product}
$M\circ N$ by
$$M\circ N=R(m+n)
\tens_{R(m)\otimes R(n)}(M\otimes N). $$

Let us denote by $R(\beta)\proj$ (respectively, $R(\beta) \gmod$)
the category of finitely generated graded projective 
(respectively, finite-dimensional over $\cor$ graded) $R(\beta)$-modules.
When $K(\R(\beta)\proj)$ and $K(\R(\beta)\gmod)$ denote the corresponding Grothendieck groups, the spaces
$$\bigoplus_{\beta \in \rootl^+} K(R(\beta)\proj) , \quad \bigoplus_{\beta \in \rootl^+} K(R(\beta)\gmod)$$
are $\A$-algebras with multiplications given by convolution products and $\A$-actions given by the grade-shift functor $q$.

A \KLR\ \emph{categorifies} the negative half of the corresponding quantum group. More precisely, we have the following theorem.

\begin{theorem}[\cite{KL09, R08}] \label{Thm:categorification}
Let $U_q(\mathfrak g)$ be the quantum group associated with a given Cartan datum $(A,P, \Pi,P^{\vee},\Pi^{\vee})$ and $R=\bigoplus_{n \geq 0}R(n)$ be the \KLR\ associated with the same Cartan datum and a matrix $(Q_{ij})_{i,j \in I}$ given in \eqref{eq:Q}.
Then there exists an $\A$-algebra isomorphism
\begin{align*}
  U^-_\A(\mathfrak g) \isoto \bigoplus_{\beta \in \rootl^+} K(R(\beta) \proj).
\end{align*}
By duality, we have
\begin{align*}
  U^-_\A(\mathfrak g)^{\vee} \isoto \bigoplus_{\beta \in \rootl^+} K(R(\beta) \gmod),
\end{align*}
where $U^-_\A(\mathfrak g)^{\vee} \seteq \set{x \in U_q^-(\g)}%
{(x, U_\A^-(\g)) \subset \A}$.
\end{theorem}

The \KLRs\ also categorify the global bases in the following sense:

\begin{theorem}[\cite{VV09, R11}] \label{thm:categorification 2}
Assume that $A$ is symmetric. 
Then under the isomorphism in Theorem \ref{Thm:categorification}, 
the lower global basis \ro respectively, upper global basis\rf\ 
 corresponds to the set of isomorphism classes of 
indecomposable projective modules \ro respectively, the set of isomorphism classes of simple modules\rf.
\end{theorem}

\section{Quantum affine algebras and their representations}

\subsection{Quantum affine algebras}
 \hfill

In this section, we briefly review the representation theory of quantum affine algebras following \cite{AK, Kas02}
and introduce a functor between the category of \KLR\ modules and the category of quantum affine algebra modules.
Hereafter, we take $\C(q)$ as the base field $\cor$.

Let $I= \{0,1,\ldots,n\}$ be an index set and $A=(a_{ij})_{i,j \in I}$ be a generalized
Cartan matrix of affine type; i.e., $A$ is positive semidefinite of corank $1$.
 Here $0$ is chosen as the leftmost vertices in the tables
in \cite[{pages 48, 49}]{Kac}. 
We take a Cartan datum $(A,P, \Pi,P^{\vee},\Pi^{\vee})$ as follows.

The coweight lattice $P^{\vee}$ is given by
\begin{align*}
  P^{\vee} = \Z h_0 \oplus \Z h_1 \oplus \cdots \oplus \Z h_n \oplus \Z d.
\end{align*}
The element $d$ is called the {\it scaling element}.
We define the simple roots $\alpha_i$'s ($i\in I$) and the fundamental weights $\Lambda_i$'s ($i\in I$) in the weight lattice $P\seteq\Hom_\Z (P^{\vee},\Z)$ as follows:
\begin{align*}
  \alpha_i(h_j)= a_{ji}, \ \alpha_i(d) = \delta_{0,i}, \ \text{and} \ \Lambda_i(h_j) = \delta_{i,j}, \ \Lambda_i(d) = 0.
\end{align*}
We denote by $\Pi=\set{\alpha_i }{i \in I}$ and 
$\Pi^{\vee}= \set{h_i }{i \in I}$ the set of simple roots and the set of simple coroots, respectively.

Let us denote by $\g$ 
the affine Kac-Moody algebra corresponding to the Cartan datum $(A,P, \Pi,P^{\vee},\Pi^{\vee})$.
Consider the positive integers $c_i$'s and $d_i$'s determined by the conditions
$$\sum_{i=0}^n c_i a_{ij}= \sum_{i=0}^n a_{ji} d_i =0
\quad\text{for all $j \in I$,}$$
and $\{c_0, c_1, \ldots, c_n \}$, $\{d_0, d_1, \ldots, d_n \}$ are relatively prime positive integers
(see \cite[Chapter 4]{Kac}). Then the center of $\g$ is 1-dimensional and is generated by the {\it canonical central element}
$$c= c_0 h_0 + c_1 h_1 + \cdots + c_n h_n$$
(\cite[Proposition 1.6]{Kac}).
Also it is known that the imaginary roots of $\g$ are nonzero integral multiples of the {\it null root}
$$\delta= d_0 \alpha_0 + d_1 \alpha_1 + \cdots + d_n \alpha_n$$
(\cite[Theorem 5.6]{Kac}).
Note that $d_0=1$ if $\g \neq A^{(2)}_{2n}$ and $d_0=2$ if $\g=A^{(2)}_{2n}$ .
Note also that $c_0 =1$ in all cases.

Now the weight lattice can be written as
$$P = \mathbb{Z} \Lambda_0 \oplus \mathbb{Z} \Lambda_1 \oplus \cdots
\oplus \mathbb{Z}\Lambda_n \oplus \mathbb{Z} d_0^{-1} \delta $$
(see \cite[Chapter 4]{Kac}).
We have
\begin{align*}
   \alpha_0 = \sum_{i=0}^n a_{ij} \Lambda_i + d_0^{-1} \delta, \quad \alpha_j = \sum_{i=0}^n a_{ij} \Lambda_i \ \text{for $j=1,\ldots, n$.}
\end{align*}

Let us denote by $U_q(\g)$ the quantum group associated with the affine Cartan datum $(A,P, \Pi,P^{\vee},\Pi^{\vee})$.
 We denote by $U_q'(\mathfrak{g})$ the subalgebra of $U_q(\mathfrak{g})$ generated by $e_i,f_i,K_i^{\pm1}(i=0,1,,\ldots,n)$ and it is called the {\it quantum affine algebra}.
Hereafter we extend the base field $\mathbb Q(q)$ of $\uqpg$ to 
$\cor\seteq\C(q)$ for convenience.

Set $$P^{\vee}_{\cl}=\mathbb{Z} h_0 \oplus \cdots \oplus \mathbb{Z}h_n \subset P^{\vee}, \ \text{and} \ \mathfrak{h}_\cl=\Q \otimes_\Z P^{\vee}_\cl \subset \mathfrak{h}.$$
Let $\cl \colon {\mathfrak h}^* \to (\mathfrak{h}_\cl)^* $ be the projection thus obtained. Then $\cl^{-1}(0)=\Q\delta$. 
We denote the {\it classical weight lattice} $\cl(P)$ by $P_\cl$. Set $\Pi_\cl = \cl(\Pi)$, and set $\Pi^{\vee}_\cl= \{h_0, \ldots, h_n\}$.
Then $U_q'(\mathfrak{g})$ can be regarded as the quantum group associated with the quintuple $(A,P_\cl, \Pi_\cl,P^{\vee}_\cl,\Pi^{\vee}_\cl)$.

Set $\h^{*0} = \{\lambda \in \h^* \,| \, \lambda(c) = 0 \} $, $\h^{*0}_\cl = \cl (\h^{*0})$
and $P^0_\cl = \cl(P) \cap \cl (\h^{*0})$.
We call the elements of $P^0_\cl$ by the {\it classical integral weight of level $0$.}
Let $W$ be the {\it Weyl group} of $\g$.
It is the subgroup of $\Aut(\h^*)$ generated by the simple reflections
$\sigma_i(\lambda)= \lambda -\lambda(h_i) \alpha_i$ for $i=0,1,\ldots,n$.
 Since $\delta(h_i)=\alpha_i(c)=0$ for $i=0,1,\ldots, n$,
there exists a group homomorphism $W \rightarrow \Aut(\h^{*0}_\cl)$.
We denote the image by $W_\cl$. Then $W_\cl$ is a finite group and it is isomorphic to the subgroup of $W$ generated by $\sigma_1, \ldots, \sigma_n$.

A $\uqpg$-module $M$ is called an {\it integrable module} if $M$ has a weight space decomposition
$$M = \bigoplus_{\lambda \in P_\cl} M_\lambda,$$
where $M_{\lambda}= \{ u \in M \ ; \ q^h u =q^{\langle h, \lambda \rangle} u \ \text{for all} \ h \in P_{\cl}^{\vee}\}$,
and if the actions of
$e_i$ and $f_i$ on $M$ are locally nilpotent for any $i\in I$.
In this paper, we mainly consider the category of finite-dimensional integrable $\uqpg$-modules.
Let us denote this category by $\CC$. The objects in this category are called 
of type $1$ (for example, see \cite{CP94}).

\begin{definition}
Let $u$ be a weight vector of weight $\lambda\in P_\cl$ of an integrable $\uqpg$-module $M$.
We call $u$ {\em extremal},
if we can find
vectors
$\{u_w\}_{w\in W}$
satisfying the following properties:
\eq
&&\text{$u_w=u$ for $w=e$,} \nonumber \\
&&
\hbox{if $\langle h_i,w\lambda\rangle\ge 0$, then
$e_iu_w=0$ and $f_i^{(\langle h_i,w\lambda\rangle)}u_w=u_{s_iw}$,} \nonumber \\
&&\hbox{if $\langle h_i,w\lambda\rangle\le 0$, then
$f_iu_w=0$ and $e_i^{(\langle h_i,w\lambda\rangle)}=u_{s_iw}$.}  \nonumber
\eneq
\end{definition}
Hence if such $\{u_w\}_{w\in W}$ exists, then it is unique and
$u_w$ has weight $w\lambda$.
We denote $u_w$ by $S_wu$.

For $\lambda \in P$,
let us denote by $W(\lambda)$
the $U_q(\g)$-module
generated by $u_\lambda$
with the defining relation that
$u_\lambda$ is an extremal vector of weight $\lambda$ (see \cite{Kas94}).
This is in fact a set of infinitely many linear relations on $u_\lambda$.

Set $\varpi_i=\Lambda_i-c_i\Lambda_0 \in P^0$ for $i=1,2,\ldots,n$.
Then $\{\cl(\varpi_i)\}_{i=1,2,\ldots,n}$ forms a basis of $P^0_\cl$.
We call $\varpi_i$ a {\it level $0$ fundamental weight}.
 As shown in \cite{Kas02}, for each $i=1,\ldots,n$ there exists a $\uqpg$-module automorphism
$z_i \colon W(\varpi_i) \rightarrow W(\varpi_i)$
which sends $u_{\varpi_i}$ to $u_{\varpi_i + \mathsf{d_i} \delta}$,
where $\mathsf{d_i} \in \Z_{>0}$ denotes the generator of the free abelian group $\{ m \in \Z \ ; \ \varpi_i + m \delta \in W \varpi_i \}$.

We define the $\uqpg$-module $V(\varpi_i)$ by
$$V(\varpi_i) = W(\varpi_i) / (z_i-1) W(\varpi_i).$$
 It can be characterized as follows(\cite[Section 1.3]{AK}):
\begin{enumerate}
\item The weights of $V(\varpi_i)$ are contained in the convex hull of $W_\cl \cl(\varpi_i)$.
\item $\dim V(\varpi_i)_{\cl(\varpi_i)} = 1$.
\item For any $\mu \in W_\cl \cl(\varpi_i) \subset P^0_\cl$, we can associate a nonzero vector $u_\mu$ of weight $\mu$ such that
$$u_{s_i \mu} = \begin{cases}
f_i^{(\langle h_i, \mu \rangle )} u_\mu & \text{if} \ \langle h_i, \mu \rangle \geq 0, \\
e_i^{(-\langle h_i, \mu \rangle) } u_\mu & \text{if} \ \langle h_i, \mu \rangle \leq 0.
\end{cases}$$
\end{enumerate}
We call $V(\varpi_i)$ the {\it fundamental representation of $\uqpg$ of weight $\varpi_i$}.

Let $-$ be an involution of a $\uqpg$-module $M$ satisfying $\ol{au}=\bar a\bar u$
for any $a\in \uqpg$ and $u\in M$.
We call such an involution a \emph{bar involution}.

We say that a finite crystal $B$ with weight in $P_\cl^0$ is a {\it simple crystal}
if there exists $\lambda \in P_\cl^0$ such that $\# (B_{\lambda})=1$ and the weight of any extremal vector of $B$ is contained $W_\cl \lambda$.

If a $\uqpg$-module $M$ has a bar involution, a crystal base with simple crystal graph, and a global base, then we say that $M$ is a {\it good module}
(\cite[Section 8]{Kas02}). 
For example, the fundamental representation $V(\varpi_i)$ 
is a good $\uqpg$-module.
Any good module is an irreducible $\uqpg$-module.

Let $A$ be a commutative $\cor$-algebra and 
let $x$ be an invertible element of $A$.
For an $A \otimes_{\cor} \uqpg$-module $M$, let us denote by $\Phi_x(M)$ the $A \otimes_{\cor} \uqpg$-module constructed in the following: there exists an 
$A$-linear bijection
$\Phi_x \colon M \rightarrow \Phi_x(M)$ which satisfy
\begin{align*}
q^h \Phi_x( u) = \Phi_x(q^h u) \quad (h \in P^{\vee}_{\cl}), &&
e_i \Phi_x( u) = x^{\delta_{i,0}} \Phi_x(e_i u), &&
f_i \Phi_x( u) = x^{-\delta_{i,0}} \Phi_x(f_i u).
\end{align*}

For invertible elements in $x, y $ of $A$ and $A \otimes_{\cor} \uqpg$-modules $M, N$, we have
$$\Phi_x \Phi_y(M) \simeq \Phi_{xy}(M)$$
by $\Phi_x(\Phi_y(u))\leftrightarrow\Phi_{xy}(u)$,
and
$$\Phi_x(M \otimes_{\cor} N) \simeq \Phi_x(M) \otimes_{\cor} \Phi_x(N)$$
by $\Phi_x(u \otimes v)\leftrightarrow  \Phi_x(u) \otimes \Phi_x(v)$.

For an integrable $\uqpg$-module $M$, the {\it affinization of $M$} is given by
$$M_{\aff} \seteq \Phi_z(\cor[z,z^{-1}]\otimes_{\cor} M).$$
Note that by defining
$q^d \Phi_z(z^n\otimes u) = q^{\lan d, n\delta\ran} \Phi_z(z^n\otimes u)$ 
for $u \in M$, $M_{\aff}$ becomes a $U_q(\g)$-module.
(We need a slight modification for $\g=A^{(2)}_{2n}$.)
For example, we have $V(\varpi_i)_\aff \simeq 
\cor[z_i^{{1 / \mathsf{d_i}}}] \otimes_{\cor[z_i]} W(\varpi_i)$, 
and hence if $\mathsf{d_i} =1$, then $W(\varpi_i) \simeq V(\varpi_i)_\aff$ \cite[Theorem 5.15]{Kas02}. 

For $a \in \cor^\times$, we define $\uqpg$-module $M_a$ by
$$M_a \seteq M_\aff / (z-a)M_\aff\simeq\Phi_a(M).$$
It is called the \emph{evaluation module of $M$ at $a$}.

\bigskip

\subsection{$R$-matrices}
 \hfill

We recall the notion of the $R$-matrices of good modules following \cite[Section 8]{Kas02}.

Let $M_1$ and $M_2$ be good $\uqpg$-modules.
Set $(M_1)_{\aff}= \Phi_{z_1}(\cor[z_1^{\pm 1}] \otimes M_1)$, 
$(M_2)_{\aff}=\Phi_{z_2}( \cor[z_2^{\pm 1}] \otimes M_2)$, and let $u_1$ and $u_2$ be the dominant extremal weight vectors in $M_1$ and $M_2$, respectively.

Then there exists a $\uqpg$-module homomorphism
\begin{equation*}
\Rnorm_{M_1, M_2} \colon (M_1)_{\aff} \otimes (M_2)_{\aff} \rightarrow
\cor(z_1,z_2)\otimes_{\cor[z_1^{\pm1},z_2^{\pm1}]} \big((M_2)_{\aff} \otimes (M_1)_{\aff} \big),\end{equation*}
satisfying
\begin{equation} \label{eq:r-matrix commute with z}
\Rnorm_{M_1, M_2} \circ z_i = z_i \circ \Rnorm_{M_1, M_2} \ \text{for} \ i=1, 2
\end{equation}
and
\begin{equation*}\Rnorm_{M_1, M_2}(u_1 \otimes u_2) = u_2 \otimes u_1
\end{equation*}
 (\cite[Section 8]{Kas02}).

Let $d_{M_1,M_2}(u) \in \cor[u]$ be a monic polynomial
with the smallest degree such that
the image of $d_{M_1,M_2}(z_1/z_2) \Rnorm_{M_1, M_2}$ 
is contained in $(M_2)_{\aff} \otimes (M_1)_{\aff}$. 
We call $\Rnorm_{M_1, M_2}$ the {\it normalized R-matrix} and $d_{M_1,M_2}$ the {\it denominator of $\Rnorm_{M_1, M_2}$}.
Since $(M_1)_{x} \otimes (M_2)_{y}$ is irreducible 
for generic $x,y\in\cor^\times$, we have
\begin{equation}\label{eq:r^2=1}
\Rnorm_{M_2, M_1} \circ \Rnorm_{M_1, M_2} = 1_{(M_1)_{\aff} \otimes (M_2)_{\aff}}.
\end{equation}

It also satisfies the Yang-Baxter equation 
\begin{equation} \label{eq:r_YB}
(\Rnorm_{M_1, M_2} \otimes 1 )\circ (1 \otimes \Rnorm_{M_1, M_3}) \circ (\Rnorm_{M_2, M_3} \otimes 1)
=(1 \otimes \Rnorm_{M_2, M_3}) \circ (\Rnorm_{M_1, M_3} \otimes 1) \circ (1 \otimes \Rnorm_{M_1, M_2}).
\end{equation}

The following fact is proved in \cite[Proposition 9.3]{Kas02}.
\Lemma\label{lem:zero}
The zeroes of $d_{M_1,M_2}(z)$ belong to
$\C[[q^{1/m}]]q^{1/m}$ for some $m\in\Z_{>0}$.
\enlemma

\begin{example} \label{ex:R-matrix}
When $\g=\widehat{\mathfrak{sl}_N}$, the normalized $R$-matrices for the fundamental representations are given as follows (see, for example, \cite{DO94}):
 \begin{align*}
   \Rnorm_{V(\varpi_k), V(\varpi_{\ell})} =
    \sum_{0 \leq i \leq \min\{k,\ell\}} \prod_{s=1}^{i} \frac{1-(-q)^{|k-\ell|+2s} z }{z- (-q)^{|k-\ell|+2s}} \
    P_{\varpi_{\max\{k,\ell\}+i}+\varpi_{\max\{k,\ell\}-i}},
 \end{align*}
where $z=z_1/z_2$ and $P_\lambda$ denotes the projection from $V(\varpi_k) \otimes V(\varpi_{\ell})$ to the direct summand $V(\lambda)$ as a $U_q(\mathfrak{sl}_N)$-module.

Note that $ \Rnorm_{V(\varpi_k), V(\varpi_{\ell})}$ has simple poles at $z=(-q)^{|k-\ell|+2s}$ for
$1 \leq s \leq \min\{k, \ell\}$.
\end{example}

\bigskip

\subsection{The action of $R^I(n)$ on $\bV$} \label{subsec:action}

Let $\{V_s\}_{s\in \mathcal{S}}$ be a family of good $\uqpg$-modules
 and  let $\lambda_s$ be a dominant extremal weight of $V_s$ and $v_s$ 
a dominant extremal weight vector in $V_s$ of weight $\lambda$.

Let $\mathbb{T}=\cor^\times$ and
let $I$ be a finite subset of $\mathcal{S} \times \mathbb{T}$. 
For each $i \in I$, let $X \colon I \rightarrow \mathbb{T}$ and 
$S \colon I \rightarrow \mathcal S$ be the maps defined by $i=(S(i), X(i))$.

For each $i,j \in I $, set
\begin{equation*}
P_{ij}(u,v) =(v-u)^{d_{ij}},
\end{equation*}
where $d_{ij}$ denotes the order of the zero of $d_{V_{S(i)},V_{S(j)}}(z_1 / z_2)$ at $z_1 / z_2 = {X(i) / X(j)}$.

Let $R^{I}$ be the \KLR\ associated with
\begin{equation}\label{eq:Q_omega} Q_{ij}(u,v) = P_{ij}(u,v)P_{ji}(v,u) \end{equation}
for $i,j \in I$.

\begin{remark}
Consider the quiver $\Gamma_I \seteq (I, \Omega)$ with the set of vertices $I$ and the set of oriented edges $\Omega$
such that
$$\# \set{h \in \Omega}{s(h) = i, t(h)=j } = d_{ij},$$
where $s(h)$ and $t(h)$ denote the source and the target of an oriented edge $h \in \Omega$.

Lemma~\ref{lem:zero} implies that $X(i)/X(j)\in\C[[q]]q$ if $d_{ij}>0$.
Hence we obtain 
$$\text{if $d_{ij} > 0$, then $d_{ji} =0$.}$$
Thus the quiver $\Gamma_I$ has neither loops nor 2-cycles.
The underlying unoriented graph of $\Gamma_I$ gives a symmetric Cartan datum and the polynomials in \eqref{eq:Q_omega} coincide with the ones used in \cite{VV09} associated with the symmetric Cartan datum of $\Gamma_I$.
\end{remark}

Set
\begin{align*}
&\bP \seteq \soplus_{\nu \in I^n} \cor [x_1, \ldots, x_n] e(\nu), \\
& \bPh\seteq \soplus_{\nu \in I^n} \Ohat_{\mathbb{T}^n, X(\nu)} e(\nu), \\
&\bK\seteq \soplus_{\nu \in I^n} \bK[\nu]e(\nu),\end{align*}
where
\begin{equation*}
\Ohat_{\mathbb{T}^n, X(\nu)} = \cor [[X_1 - X(\nu_1), \ldots, X_n-X(\nu_n)]]
\end{equation*}
 is the completion of the local ring of $\mathbb{T}^n$ at $X(\nu)\seteq(X(\nu_1),\ldots,X(\nu_n))$
 and $\bK[\nu]$ is the field of quotients of $\Ohat_{\mathbb{T}^n, X(\nu)}$.

Then we have
\begin{equation*}
\bP\hookrightarrow\bPh\hookrightarrow\bK
\end{equation*}
as $\cor$-algebras, where the first arrow is given by
\begin{equation*}
x_k e(\nu) \mapsto (X(\nu_k)^{-1}X_k -1) e(\nu).
\end{equation*}
Note that
\begin{equation*}
\cor [X_1^{\pm1}, \ldots, X_n^{\pm1}] \subset \Ohat_{\mathbb{T}^n, X(\nu)} \ \text{for all} \ \nu \in I^n.
\end{equation*}
Let
\begin{align*}
\cor [S_n] \seteq \bigoplus_{w \in S_n} \cor r_w
\end{align*}
be the group algebra of $S_n$;
i.e., the $\cor$-algebra with the defining relations
\begin{align} \label{eq:rel of S_n}
&r_a^2 =1 && a = 1, \ldots, n-1  \nonumber \\
&r_w r_{w'} = r_{ww'}, \\
&r_a r_{a+1} r_a = r_{a+1} r_a r_{a+1} && a=1,\ldots,n-2 \nonumber
\end{align}
where $r_a = r_{s_a} \ (1 \leq a <n)$.

The symmetric group $S_n$ acts on $\bP$, $\bPh$, $\bK$ 
from the left and we have
\begin{equation*}
\bP\otimes {\cor[S_n]} \hookrightarrow \bPh\otimes {\cor[S_n]} \hookrightarrow \bK\otimes {\cor[S_n]} 
\end{equation*}
as algebras. Here the algebra structure on $\bK\otimes {\cor[S_n]}$ is given by
\begin{align} \label{eq:rel of Khat S_n}
r_w f =w(f) r_w & \quad \text{for $f \in \bK$, $w \in S_n$.}
\end{align}

Then $\bK$
may be regarded as a right $\bK \otimes {\cor[S_n]}$-module 
by
$a (f\otimes r_w)=w^{-1}(af)$ ($a,f\in \bK$ and $w\in S_n$).

Set
\begin{equation*}
e(\nu) \tau_a = \begin{cases}
e(\nu) r_a P_{\nu_a, \nu_{a+1}}(x_{a+1}, x_a) & \text{if} \ \nu_a \neq \nu_{a+1} \\
e(\nu) (r_a -1)(x_a -x_{a+1})^{-1} & \text{if} \ \nu_a = \nu_{a+1}.
\end{cases}
\end{equation*}

Then the subalgebra of $\bK\otimes {\cor[S_n]}$ generated by
\begin{align*}
e(\nu) \ (\nu \in J^n), && x_a \ (1 \leq a \leq n), && e(\nu)\tau_a \ (1 \leq a \leq n-1)
\end{align*}
is isomorphic to the \KLR\ $R^I(n)$ of degree $n$ associated with $Q_{ij}(u,v) = P_{ij}(u,v) P_{ji}(v,u)$.
\cite[Proposition 3.12]{R08}, \cite[Theorem 2.5]{KL09}.

\bigskip
For each $\nu =(\nu_1,\ldots, \nu_n) \in I^n$, we set
$$
V_\nu = \Phi_{X_1}(\cor[X_1^{\pm1}]\otimes V_{S(\nu_1)}) \otimes \cdots
\otimes \Phi_{X_n}(\cor[X_n^{\pm1}]\otimes V_{S(\nu_n)})$$
which is a $\cor[X_1^{\pm1},\ldots, X_n^{\pm1}]\otimes\uqpg$-module.
Then we define
\begin{align}
\ba{ll}
\bV&\seteq \soplus_{\nu \in I^n} 
\Ohat_{\mathbb{T}^n, X(\nu)}\otimes _{\cor[X_1^{\pm1},\ldots,X_n^{\pm1}]}V_\nu e(\nu), \\
\bVK&\seteq \bK\tens_{\bP}\bV.\ea
\label{eq:Vhat} 
\end{align}

For each $\nu \in I^n$ and $a=1,\ldots, n-1$, there exists a $\uqpg$-module homomorphism
\begin{equation*}
R^{\nu}_{a, a+1} \colon V_\nu \rightarrow
\cor(X_1,\ldots,X_n) \otimes_{\cor[X_1^{\pm1}, \ldots X_n^{\pm1}]} V_{s_a(\nu)}
\end{equation*}
which is given by
\begin{equation*}
 v_1 \otimes \cdots \otimes v_a \otimes v_{a+1} \otimes \cdots \otimes v_n
\mapsto v_1 \otimes \cdots \otimes \Rnorm_{V_{S(\nu_a)},V_{S(\nu_{a+1})}}(v_a \otimes v_{a+1}) \otimes \cdots \otimes v_n
\end{equation*}
for $v_k \in \Phi_{X_k}(V_{S(\nu_k)})$ $(1 \leq k \leq n)$.

It follows that
\begin{align}
&R^{\nu}_{a, a+1} \circ X_k = X_{s_a(k)} \circ R^{\nu}_{a, a+1} && \text{from \eqref{eq:r-matrix commute with z},} \nonumber\\
&R^{s_a(\nu)}_{a, a+1} \circ R^{\nu}_{a, a+1} = 1_{V_{\nu}}&& \text{from \eqref{eq:r^2=1},} \nonumber \\
&R^{s_{a+1}s_a(\nu)}_{a, a+1} \circ R^{s_a (\nu)}_{a+1, a+2} \circ R^{\nu}_{a, a+1} =
R^{s_a s_{a+1}(\nu)}_{a+1, a+2} \circ R^{s_{a+1}(\nu)}_{a, a+1} \circ R^{\nu}_{a+1, a+2}
&& \text{from \eqref{eq:r_YB}.} \ \nonumber
\end{align}
Set $d_{\nu_a,\nu_{a+1}}(u)=d_{V_{S(\nu_a)}, V_{S(\nu_{a+1})}} (u)$. Then,
\begin{equation*}
d_{\nu_a,\nu_{a+1}}(X_{a+1}/X_a) R^{\nu}_{a, a+1} \colon V_{\nu} \rightarrow V_{s_a(\nu)}.
\end{equation*}

The algebra $\bK \otimes {\cor[S_n]}$ acts on $\bVK$ from the right, where
\begin{align*}
e(\nu) r_a : &\widehat{\mathbb K}_\nu \otimes_{\cor[X_1^{\pm1}, \ldots X_n^{\pm1}]} V_\nu \\
&\rightarrow \widehat{\mathbb K}_{s_a(\nu)} \otimes_{\cor(X_1,\ldots,X_n)} \big( \cor(X_1,\ldots,X_n) \otimes_{\cor[X_1^{\pm1}, \ldots X_n^{\pm1}]} V_{s_a(\nu)}\big)
\end{align*}
is given by
\begin{equation*}
(f \otimes v) e(\nu) r_a = s_a(f) e(s_a(\nu)) \otimes R^{\nu}_{a, a+1}(v)
\end{equation*}
for $f \in \widehat{\mathbb K}_{\nu}$, $v \in \cor(X_1,\ldots,X_n) \otimes_{\cor[X_1^{\pm1}, \ldots, X_n^{\pm1}]} V_\nu$.
The subalgebra $\widehat{\mathbb{K}}_n$ acts by the multiplication.
The relations \eqref{eq:rel of S_n} and \eqref{eq:rel of Khat S_n} follow from the properties of normalized $R$-matrices and hence we have a well-defined action of the algebra $\widehat{\mathbb{K}}_n \otimes {\cor[S_n]}$ on $\Vhat^n$.
Since the normalized $R$-matrices are $\uqpg$-module homomorphisms, 
the right action of $\widehat{\mathbb{K}}_n \otimes {\cor[S_n]}$ 
commutes with the left action of $\uqpg$ on $\Vhat^n$.

\begin{theorem}
The subspace $\Vhat^n$ of $\Vhat_K^n$ is stable under the action of the subalgebra $R^I(n)$ of $\widehat{\mathbb{K}}_n \otimes {\cor[S_n]}$.
In particular, $\Vhat^n$ has a structure of
$(\uqpg, R^I(n))$-bimodule.
\begin{proof}
It is obvious that $\Vhat^n$
is stable by the actions of $e(\nu)$ $(\nu \in I^n)$ and 
$x_a$ $(1 \leq a \leq n)$.
Thus it is enough to show that $\Vhat^n$ is stable under $e(\nu) \tau_a$ $(\nu \in I^n, \ 1 \leq a < n)$.

Assume $\nu_a \neq \nu_{a+1}$. Then we have
\begin{align}
&e(\nu) r_a P_{\nu_a,\nu_{a+1}}(x_{a+1},x_{a}) \nonumber\\
=&e(\nu) r_a d_{\nu_a,\nu_{a+1}}(X_{a+1}/X_a)
\Big( e(s_a(\nu)) \dfrac{ P_{\nu_a,\nu_{a+1}}(x_{a+1},x_{a})}{d_{\nu_a,\nu_{a+1}}(X_{a+1}/X_a)}
\Big) \nonumber\\
=&e(\nu) r_a d_{\nu_a,\nu_{a+1}}(X_{a+1}/X_a)
\Big( e(s_a(\nu)) \dfrac{ (X(\nu_{a+1})^{-1}X_a-X(\nu_a)^{-1}X_{a+1})^{d_{\nu_a, \nu_{a+1}}}} {d_{\nu_a,\nu_{a+1}}(X_{a+1}/X_a)} \Big)\nonumber.
\end{align}
Since $d_{\nu_a, \nu_{a+1}}$ is the multiplicity of the zero of the polynomial
$d_{\nu_a,\nu_{a+1}}(X_{a+1}/X_a)$ at $X_{a+1}/X_a=X(\nu_a) / X(\nu_{a+1})$,
we have
$$\dfrac{ (X(\nu_{a+1})^{-1}X_a-X(\nu_a)^{-1}X_{a+1})^{d_{\nu_a, \nu_{a+1}}}} {d_{\nu_a,\nu_{a+1}}(X_{a+1}/X_a)} \in \Ohat_{\mathbb{T}^n, X(s_a(\nu))}.$$

It follows that
\begin{align*}
&\Big(\Ohat_{\mathbb T^n, X(\nu) } \otimes_{\cor[X_1^{\pm1}, \ldots X_n^{\pm1}]} V_\nu \Big)
e(\nu) \tau_a \\
=&\Big(\big(\Ohat_{\mathbb T^n, X(\nu) } \otimes_{\cor[X_1^{\pm1}, \ldots X_n^{\pm1}]} V_\nu \big)
e(\nu) r_a \Big) P_{\nu_a,\nu_{a+1} }(x_{a+1}, x_a) \\
\subset &\Big(\Ohat_{\mathbb T^n, X(s_a(\nu)) } \otimes_{\cor[X_1^{\pm1}, \ldots X_n^{\pm1}]} V_{s_a(\nu)}
 \Big) \dfrac{(X(\nu_{a+1})^{-1}X_a-X(\nu_a)^{-1}X_{a+1})^{d_{\nu_a, \nu_{a+1}}}}%
{d_{\nu_a, \nu_{a+1}}(X_{a+1} / X_a)} \\
\subset & \Ohat_{\mathbb T^n, X(s_a(\nu)) } \otimes_{\cor[X_1^{\pm1}, \ldots X_n^{\pm1}]} V_{s_a(\nu)},
\end{align*}
as desired.

Assume $\nu_a=\nu_{a+1}$.
Then $\Rnorm_{V_{S(\nu_a)},V_{S(\nu_a)}}$ does not have a pole at
$X_a=X_{a+1}$ by Lemma~\ref{lem:zero}.
Since $\Phi_x(V_{S(\nu_a)})\otimes \Phi_x(V_{S(\nu_{a})})$ is irreducible for any $x\in\cor^\times$, we obtain
$\Rnorm_{V_{S(\nu_a)},V_{S(\nu_a)}}\vert_{X_a=X_{a+1}}=\id$.
Therefore, we have
\begin{align*}
&\Big(\Ohat_{\mathbb T^n, X(\nu) } \otimes_{\cor[X_1^{\pm1}, \ldots X_n^{\pm1}]} V_\nu \Big)
e(\nu) \tau_a \\
=&\Big(\Ohat_{\mathbb T^n, X(\nu) } \otimes_{\cor[X_1^{\pm1}, \ldots X_n^{\pm1}]} V_\nu \Big)
e(\nu) (r_a -1)(x_a -x_{a+1})^{-1} \\
=&\Big(\Ohat_{\mathbb T^n, X(\nu) } \otimes_{\cor[X_1^{\pm1}, \ldots X_n^{\pm1}]} V_\nu \Big)
e(\nu) X(\nu_a) (r_a -1) (X_a -X_{a+1})^{-1} \\
\subset & \Ohat_{\mathbb T^n, X(\nu) } \otimes_{\cor[X_1^{\pm1}, \ldots X_n^{\pm1}]} V_{\nu},
\end{align*}
as desired.
\end{proof}
\end{theorem}

Since $\Vhat^n$ is a $(\uqpg, R^I(n))$-bimodule, we can construct the following functor:
\begin{align*}
\F_n \colon R^I(n) \gmod &\rightarrow \uqpg \modl \label{eq:the functor}\\
M &\mapsto \F_n(M) \seteq \Vhat^n \otimes_{R^I(n)} M,
\end{align*}
where $\uqpg \modl$ denotes the category of finite-dimensional $\uqpg$-modules.

\begin{theorem} \label{thm:conv to tensor}
Let $M_1 \in R^I(n_1) \gmod$ and $M_2 \in R^I(n_2) \gmod$.
Then there exists a canonical isomorphism of $\uqpg$-modules
\begin{equation*}
\F_{n}(M_1 \circ M_2) \simeq \F_{n_1}(M_1) \otimes \F_{n_2}(M_2),
\end{equation*}
where $n=n_1+n_2$.
\end{theorem}

\begin{proof}
For each $\nu=(\nu_1,\ldots,\nu_n) \in I^{n}$, set
$\nu'=(\nu_1, \ldots, \nu_{n_1})$ and $\nu''=(\nu_{n_1+1}, \ldots, \nu_{n})$.
Then we have an algebra homomorphism
$\Ohat_{\mathbb{T}^{n_1}, X(\nu')} \otimes \Ohat_{\mathbb T^{n_2}, X(\nu'')}
\to  \Ohat_{\mathbb T^n, X(\nu)}$.
Moreover, for any finite-dimensional $\Ohat_{\mathbb{T}^{n_1}, X(\nu')}$-module
$L_1$ and any finite-dimensional $\Ohat_{\mathbb{T}^{n_1}, X(\nu'')}$-module $L_2$,
the induced morphism 
$$L_1\otimes L_2\to \Ohat_{\mathbb T^n, X(\nu)}
\tens_{\Ohat_{\mathbb{T}^{n_1}, X(\nu')} \otimes \Ohat_{\mathbb T^{n_2}, X(\nu'')}}(L_1\otimes L_2)$$
is an isomorphism.
Hence for any finite-dimensional $\mathbb{P}_{n_1}$-module
$L_1$ and any finite-dimensional $\mathbb{P}_{n_2}$-module $L_2$,
the induced morphism 
$$
(\Vhat^{n_1}\otimes \Vhat^{n_2})\tens_{\mathbb{P}_{n_1}\otimes\,\mathbb{P}_{n_2}}(L_1\otimes L_2)
\to
\Vhat^n\tens_{\mathbb{P}_{n_1}\otimes\,\mathbb{P}_{n_2}}(L_1\otimes L_2)
$$
is an isomorphism.

The module
$\Vhat^n \otimes_{R^I(n)}(M_1 \circ M_2)\simeq \Vhat^n
\otimes_{R^I(n_1) \otimes R^I(n_2)} (M_1 \otimes M_2)$
is the quotient of
$\Vhat^n\otimes_{\mathbb{P}_{n_1}\otimes\mathbb{P}_{n_2}} (M_1 \otimes M_2)$
by the submodule generated by
$va\otimes u-v\otimes au$  where $a\in R^I(n_1) \otimes R^I(n_2)$,
$v\in \Vhat^n$, $u\in M_1 \otimes M_2$.
A similar result holds also for $\bl\Vhat^{n_1}\otimes 
\Vhat^{n_2}\br\otimes_{R^I(n_1) \otimes R^I(n_2)} (M_1 \otimes M_2)$.
Thus we obtain the desired result
$$\bl\Vhat^{n_1}\otimes \Vhat^{n_2}\br\otimes_{R^I(n_1) \otimes R^I(n_2)} (M_1 \otimes M_2)
\simeq \Vhat^n\otimes_{R^I(n_1) \otimes R^I(n_2)} (M_1 \otimes M_2).$$
\end{proof}

The following propositions are key ingredients for proving our main theorem.
\begin{prop}[{\cite[Corollary 2.9]{Kato12}, \cite[Theorem 4.6]{McNa12}}]%
\label{pro:finite global dimension}
If the quiver associated with $R^I(n)$ is of type $A, D, E$, then $R^I(n)$ has finite global dimension.
\end{prop}

\begin{prop} \label{pro:projectivity}
Let $A\to B$ be a homomorphism of algebras.
We assume the following conditions:
\bna
\item $B$ is a finitely generated projective $A$-module,
\item $\Hom_A(B,A)$ is a projective $B$-module,
\item the global dimension of $B$ is finite.
\ee
Then we have:
\bni
\item any $B$-module projective over $A$ is projective over $B$,
\item any $B$-module flat over $A$ is flat over $B$.
\ee
\end{prop}

\begin{proof}
Since the proof is similar, we give only the proof of (ii).

Let us denote by $\fd M$ the flat dimension of an $A$-module $M$.
By (a) we have $$\fd(M)\le\fd[B](M)$$ for any $B$-module $M$.

By (b),  $\Hom_A(B,A)\tens_AL$ is a flat $B$-module
if $L$ is a flat $A$-module. Indeed, 
the functor $X\tens_B\Hom_A(B,A)$ is exact in $X\in\Mod(A^\opp)$
and hence $X\tens_B\Hom_A(B,A)\tens_AL$
is also exact in $X$.

On the other hand, for any $A$-module $L$, the canonical $B$-module homomorphism
$$\Hom_A(B,A)\tens_AL\to \Hom_A(B,L),
\qquad f\tens s\longmapsto (B\ni b\mapsto f(b)s)
$$
is an isomorphism by (a). 
Hence we conclude that
$\Hom_A(B,L)$ is a flat $B$-module
for any flat $A$-module $L$.
It immediately implies that
$$\fd[B]\bl\Hom_A(B,L)\br\le\fd (L)\quad\text{for any $A$-module $L$.}$$

Now, let $M$ be a $B$-module.
Then there exists a canonical $B$-module homomorphism
$$\vphi_M\colon M\to \Hom_A(B,M)$$
given by $\vphi_M(x)(b)=bx$.
It is evidently injective.

In order to prove the proposition,
it is enough to show the following statement for any $d\ge0$:
\eq \text{for any $B$-module $M$,
$\fd (M)\le d$ implies $\fd[B](M)\le d$.} \nonumber
\eneq
We shall show it by the descending induction on $d$.
If $d\gg0$, it is a consequence of (c).
Let $M$ be a $B$-module with $\fd(M)\le d$.
We have an exact sequence
$$0\to M\To[\vphi_M] \Hom_A(B,M)\to N\to 0.$$
Then
$\fd\bl\Hom_A(B,M)\br\le\fd[B]\bl\Hom_A(B,M)\br\le \fd (M)\le d$.
Hence we have $\fd[A]N\le d+1$,
which implies that $\fd[B](N)\le d+1$
by the induction hypothesis.
Finally we conclude that
$\fd[B](M)\le d$.
Thus the induction proceeds.
\end{proof}

\begin{theorem} \label{thm:exact}
If the quiver associated with $R^I(n)$ is of type $A, D, E$, then the functor $\F_n$ is exact.
\end{theorem}
\begin{proof}
Let us apply Proposition \ref{pro:projectivity} with $A=\mathbb{P}_n$
and $B= R^I(n)$.
The conditions (a) and (b) are well-known, and (c) is nothing but Proposition \ref{pro:finite global dimension}. Therefore, since $\Vhat^n$ is a flat $\mathbb{P}_n$-module,
it is a flat $R^I(n)$-module.
\end{proof}

\bigskip
\end{document}